\documentclass[11pt]{amsart}
\usepackage{amsmath,amsthm,amsfonts,amscd,amssymb,eucal,latexsym,mathrsfs,enumerate,framed}

\usepackage{hyperref}

\newcommand{\Cu}{\mathrm{Cu}}
\newcommand{\CuT}{\mathrm{Cu}_{\mathbb T}}
\newcommand{\Ell}{\mathrm{Ell}}
\newcommand{\Inv}{\mathbf{Inv}}
\newcommand{\Sc}{\Sigma}
\newcommand{\eps}{\varepsilon}
\newcommand{\Bb}{\mathcal B}
\newcommand{\Hs}{\mathcal H}
\newcommand{\Ks}{\mathcal K}
\newcommand{\id}{\mathrm{id}}
\newcommand{\co}{\overline{\mathrm{co}}^w}
\newcommand{\ad}{\mathrm{ad}}
\newcommand{\Z}{\mathcal Z}
\newcommand{\cb}{\mathrm{cb}}
\newcommand{\K}{\mathcal K}
\newcommand{\M}{\mathcal M}
\newcommand{\N}{\mathcal N}
\numberwithin{equation}{section}

\theoremstyle{plain}
\newtheorem{lemma}{Lemma}[section]
\newtheorem{theorem}[lemma]{Theorem}
\newtheorem{corollary}[lemma]{Corollary}
\newtheorem{proposition}[lemma]{Proposition}
\newtheorem*{proposition*}{Proposition}

\theoremstyle{definition}
\newtheorem{definition}[lemma]{Definition}

\theoremstyle{remark}

\title{The Cuntz semigroup and stability of close $C^*$-algebras}
\author[F.~Perera]{Francesc Perera}
\address{\hskip-\parindent
Francesc Perera, Department of Mathematics, 08193 Bellaterra, Barcelona, Spain.}
\email{perera@mat.uab.cat }
\author[A.~Toms]{Andrew Toms}
\address{\hskip-\parindent
Andrew Toms, Department of Mathematics, Purdue University, 150 North University Street, West Lafayette, IN 47907, USA.}
\email{atoms@pudue.edu}
\author[S.~White]{Stuart White}
\address{\hskip-\parindent
Stuart White, School of Mathematics and Statistics, University of Glasgow, 
University Gardens, Glasgow Q12 8QW, Scotland.}
\email{stuart.white@glasgow.ac.uk}
\author[W.~Winter]{Wilhelm Winter}
\address{\hskip-\parindent
Wilhelm Winter, Mathematisches Institut der WWU M\"unster, Einsteinstra\ss{}e 62, 48149, M\"unster, Germany.}
\email{wwinter@uni-muenster.de}

\thanks{Research partially supported by  EPSRC (grants No.\  EP/G014019/1 and No.\ EP/I019227/1), by the DFG (SFB 878), by NSF (DMS-0969246), by  the DGI MICIIN (grant No.\ MTM2011-28992-C02-01), and by the Comissionat per Universitats i Recerca de la Generalitat de Catalunya. Andrew Toms is partially supported by the 2011 AMS Centennial Fellowship.}

\date{\today}

\begin{document}

\begin{abstract}
We prove that separable $C^*$-algebras which are completely close in a natural uniform sense have isomorphic Cuntz semigroups, continuing a line of research developed by Kadison-Kastler, Christensen, and Khoshkam.  This result has several applications:  we are able to prove that the property of stability is preserved by close $C^*$-algebras provided that one algebra has stable rank one;  close $C^*$-algebras must have affinely homeomorphic spaces of lower-semicontinuous quasitraces;  strict comparison is preserved by sufficient closeness of $C^*$-algebras.  We also examine $C^*$-algebras which have a positive answer to Kadison's Similarity Problem,  as these algebras are completely close whenever they are close.  A sample consequence is that sufficiently close $C^*$-algebras have isomorphic Cuntz semigroups when one algebra absorbs the Jiang-Su algebra tensorially.
\end{abstract}

\maketitle

\section{introduction}

In 1972 Kadison and Kastler introduced a metric $d$ on the $C^*$-subalgebras of a given $C^*$-algebra by equipping the unit balls of the subalgebras with the Hausdorff metric (in norm) (\cite{KK:AJM}).  They conjectured that sufficiently close $C^*$-subalgebras of $\mathcal B(\Hs)$ should be isomorphic, and this conjecture was recently established by Christensen, Sinclair, Smith and the last two named authors (\cite{CSSWW:Acta}) when one $C^*$-algebra is separable and nuclear.  The one-sided version of this result---that a sufficiently close near inclusion of a nuclear separable $C^*$-algebra into another $C^*$-algebra gives rise to a true inclusion---was later proved by Hirshberg, Kirchberg, and the third named author (\cite{HKW:ADV}).  These results and others (see \cite{CSSW:GAFA}, \cite{CCSSWW:PNAS}) have given new momentum to the perturbation theory of operator algebras.
 
The foundational paper \cite{KK:AJM} was concerned with structural properties of close algebras, showing that the type decomposition of a von Neumann algebra transfers to nearby algebras. We continue this theme here asking ``Which properties or invariants of $C^*$-algebras are preserved by small perturbations?" With the proof of the Kadison-Kastler conjecture the answer for nuclear separable $C^*$-algebras is, ``All of them."  Here we consider general separable $C^*$-algebras where already, there are some results. Sufficiently close $C^*$-algebras have isomorphic lattices of ideals (\cite{P:IJM}) and algebras whose stabilizations are sufficiently close have isomorphic K-theories (\cite{K:JOT}). This was extended to the Elliott invariant consisting of K-theory, traces, and their natural pairing, in \cite{CSSW:GAFA}.  A natural next step is to consider the Cuntz semigroup of (equivalence classes of) positive elements (in the stabilisation) of a $C^*$-algebra, due both to its exceptional sensitivity in determining non-isomorphism (\cite{T:Ann}), classification results using the semigroup (\cite{R:Adv}) and the host of $C^*$-algebraic properties that can be formulated as order-theoretic properties of the semigroup: for example there is strong evidence to suggest that the behaviour of the Cuntz semigroup characterises important algebraic regularity properties of simple separable nuclear $C^*$-algebras (\cite{MS:Acta,W:Invent1,W:Invent2}).  We prove that algebras whose stabilizations are sufficiently close do indeed have isomorphic Cuntz semigroups, a surprising fact given the sensitivity of a Cuntz class to perturbations of its representing positive element.  This is in stark contrast with the case of Murray-von Neumann equivalence classes of projections, where classes are stable under perturbations of the representing projection of size strictly less than one.  The bridge between these two situations is that we can arrange for the representing positive element of a Cuntz class to be almost a projection in trace.  We exploit this fact through the introduction of what we call {\it very rapidly increasing sequences} of positive contractions, increasing sequences where each element almost acts as a unit on its predecessor.    

The Kadison-Kastler metric $d$ is equivalent to a complete version $d_{\cb}$ (given by applying $d$ to the stabilisations) if and only if Kadison's Similarity Problem has a positive solution \cite{CSSW:GAFA,CCSSWW:arXiv}; the latter is known to hold in considerable generality, for instance in the case of $\mathcal{Z}$-stable algebras (\cite{JW:BLMS}).  We show how this result, and a number of other similarity results for $C^*$-algebras, can be put in a common framework using Christensen's property $D_k$ (\cite{C:IJM}), and, building on \cite{CSSW:GAFA}, make a more careful study of automatic complete closeness and its relation to property $D_k$.  We prove that  if an algebra $A$ has $D_k$ for some $k$, then $d(A \otimes \mathcal{K}, B \otimes \mathcal{K})\leq C(k) d(A,B)$, where $C(k)$ is a constant independent of $A$ and $B$; as a consequence sufficiently close $C^*$-algebras have isomorphic Cuntz semigroups provided one algebra is $\Z$-stable. 

Stability is perhaps the most basic property one could study in perturbation theory, yet proving its permanence under small perturbations has seen very little progress.  We take a significant step here by proving that stability is indeed preserved provided that one of the algebras considered has stable rank one.  The proof is an application of our permanence result for the Cuntz semigroup.  Another application is our proof that stably close $C^*$-algebras have affinely homeomorphic spaces of lower semicontinuous 2-quasitraces. This extends and improves an earlier results from \cite{CSSW:GAFA} showing that the affine isomorphism between the trace spaces of stably close $C^*$-algebras obtained in \cite{CSSW:GAFA} is weak$^*$-weak$^*$ continuous.

The paper is organized as follows:  Section \ref{Sect2} contains the preliminaries on the Cuntz semigroup and the Kadison-Kastler metric;  Section \ref{Sect3} establishes the permanence of the Cuntz semigroup under complete closeness;  Section \ref{Sect4} discusses property $D_k$ and proves our permanence result for stability;  Section \ref{Sect5} proves permanence for quasitraces.

\section{Preliminaries}\label{Sect2}

Throughout the paper we write $A^+$ for the positive elements of a $C^*$-algebra $A$, $A_1$ for the unit ball of $A$ and $A^+_1$ for the positive contractions in $A$.  

In the next two subsections we review the definition and basic properties of the Cuntz semigroup. A complete account can be found in the survey \cite{APT:Contemp}.

\subsection{The Cuntz semigroup} Let $A$ be a $C^*$-algebra. Let us consider on
$(A\otimes\mathcal K)^+$ the relation $a\precsim b$ if $v_nbv_n^*\to a$ for some sequence $(v_n)$ in $A\otimes\mathcal K$.
Let us write $a\sim b$ if $a\precsim b$ and $b\precsim a$. In this case we say that $a$ is \emph{Cuntz equivalent} to $b$.
Let $\Cu(A)$ denote the set $(A\otimes \mathcal K)^+/\sim$ of Cuntz equivalence classes. We use $\langle a \rangle$ to denote the class of $a$ in $\Cu(A)$.  It is clear that
$\langle a \rangle\leq \langle b \rangle \Leftrightarrow a\precsim b$ defines an order on $\Cu(A)$. We also endow $\Cu(A)$
with an addition operation by setting $\langle a \rangle+\langle b \rangle:=\langle a'+b' \rangle$, where
$a'$ and $b'$ are orthogonal and Cuntz equivalent to $a$ and $b$ respectively (the choice of $a'$ and $b'$
does not affect the Cuntz class of their sum). The semigroup $W(A)$ is then the subsemigroup of $\Cu(A)$ of Cuntz classes
with a representative in $\bigcup_n M_n(A)_+$.

Alternatively, $\Cu(A)$ can be defined to consist of equivalence classes of countably generated Hilbert modules over $A$ \cite{CEI:Crelle}.  The equivalence relation boils down to isomorphism in the case that $A$ has stable rank one, but is rather more complicated in general and as we do not require the precise definition of this relation in the sequel, we omit it.  We note, however, that the identification of these two approaches to $\Cu(A)$ is achieved by associating the element $\langle a \rangle$ to the class of the Hilbert module $\overline{a\ell_2(A)}$. 

\subsection{The category $\mathbf{Cu}$}\label{cu}

The semigroup $\Cu(A)$ is an object in a category of ordered Abelian monoids denoted by ${\mathbf{Cu}}$ introduced in \cite{CEI:Crelle} with additional properties.  Before stating them, we require the notion of order-theoretic compact containment.  Let $T$ be a pre-ordered set with $x,y \in T$.  We say that $x$ is \emph{compactly contained} in $y$---denoted by $x \ll y$---if for any increasing sequence $(y_n)$ in $T$ with supremum $y$, we have $x \leq y_{n_0}$ for some $n_0 \in \mathbb{N}$. An object $S$ of $\mathbf{Cu}$ enjoys the following properties (see \cite{CEI:Crelle,APT:Contemp}), which we use repeatedly in the sequel. In particular the existence of suprema in property {\bf P3} is a crucial in our construction of a map between the Cuntz semigroups of stably close $C^*$-algebras.
\begin{enumerate}
\item[{\bf P1}] $S$ contains a zero element; \vspace{1mm}
\item[{\bf P2}] the order on $S$ is compatible with addition:  $x_1 + x_2
\leq y_1 + y_2$ whenever $x_i \leq y_i , i \in \{1, 2\}$; \vspace{1mm}
\item[{\bf P3}] every countable upward directed set in $S$ has a supremum; \vspace{1mm}
\item[{\bf P4}] for each $x\in S$, the set $x_\ll = \{y \in S \ |  \ y \ll x\}$ is upward directed with respect to both
$\leq$ and $\ll$, and contains a sequence $(x_n)$ such that
$x_n \ll x_{n+1}$ for every $n \in \mathbb{N}$ and $\sup_n x_n = x$; \vspace{1mm}
\item[{\bf P5}] the operation of passing to the supremum of a countable upward directed set and the relation
$\ll$ are compatible with addition:  if $S_1$ and $S_2$ are countable upward directed sets in $S$, then
$S_1 + S_2$ is upward directed and $\sup (S_1 + S_2) = \sup S_1 + \sup S_2$, and if $x_i \ll y_i$ for $i \in \{1, 2\}$, then
$x_1 + x_2 \ll y_1 + y_2$ .
\end{enumerate}

\noindent
We say that a sequence $(x_n)$ in $S \in \textbf{Cu}$ is {\it rapidly increasing} if $x_n \ll x_{n+1}$ for all $n$.
We take the scale $\Sc(\Cu(A))$ to be the subset of $\Cu(A)$ obtained as supremums of increasing sequences from $A^+$.

For objects $S$ and $T$ from $\textbf{Cu}$, the map $\phi\colon S\to T$ is a morphism in the category $\mathbf{Cu}$ if
\begin{enumerate}
\item[{\bf M1}] $\phi$ is order preserving; \vspace{1mm}
\item[{\bf M2}] $\phi$ is additive and maps 0 to 0; \vspace{1mm}
\item[{\bf M3}] $\phi$ preserves the suprema of increasing sequences; \vspace{1mm}
\item[{\bf M4}] $\phi$ preserves the relation $\ll$. \vspace{1mm}
\end{enumerate}

\subsection{The Kadison-Kastler metric}

Let us recall the definition of the metric on the collection of all $C^*$-subalgebras of a $C^*$-algebra introduced in \cite{KK:AJM}.

\begin{definition}\label{close}
Let $A,B$ be $C^*$-subalgebras of a $C^*$-algebra $C$.  Define a metric $d$ on all such pairs as follows:  $d(A,B)<\gamma$ if and only if for each $x$ in the unit ball of $A$ or $B$, there is $y$ in the unit ball of the other algebra such that $\|x-y\|<\gamma$.
\end{definition}

In this definition, we typically take $C=\mathcal B(\mathcal{H})$ for a Hilbert space $\mathcal{H}$.  The complete, or stabilised version, of the Kadison-Kastler metric is defined by $d_{\cb}(A,B)=d(A\otimes \mathcal{K},B \otimes \mathcal{K})$ inside $C \otimes \mathcal{K}$ (here $\mathcal{K}$ is the compact operators on $\ell^2(\mathbb{N}))$; the notion $d_{\cb}$ is used for this metric as $d_{\cb}(A,B)\leq\gamma$ is equivalent to the condition that $d(\mathrm{M}_n(A), \mathrm{M}_n(B))\leq\gamma$ for every $n$.  

We repeatedly use the standard fact that if $d(A,B)<\gamma$, then given a positive contraction $a\in A_1^+$, there exists a positive contraction $b\in B_1^+$ with $\|a-b\|<2\alpha$.  One way of seeing this is to use the hypothesis $d(A,B)<\gamma$ to approximate $a^{1/2}$ by some $c\in B_1$ with $\|a^{1/2}-c\|<\gamma$. Then take $b=cc^*$ so that $$
\|a-b\|\leq \|a^{1/2}(a^{1/2}-c)\|+\|(a^{1/2}-c^*)c\|<2\gamma.
$$

There is also a one-sided version of closeness introduced by Christensen in \cite{C:Acta}, which is referred to as a {\it $\gamma$-near inclusion}:

\begin{definition}\label{closeinclusion}
Let $A,B$ be $C^*$-subalgebras of a $C^*$-algebra $C$ and let $\gamma>0$.  Write $A \subseteq_\gamma B$ if for every $x$ in the unit ball of $B$, there is $y \in B$ such that $\|x-y\|\leq \gamma$ (note that $y$ need not be in the unit ball of $B$).  Write $A\subset_\gamma B$ if there exists $\gamma'<\gamma$ with $A\subseteq_{\gamma'}B$.  As with the Kadison-Kastler metric, we also use complete, or stabilised, near inclusions: write $A\subseteq_{\mathrm{cb},\gamma}B$ when $A\otimes M_n\subseteq_{\gamma}B\otimes M_n$ for all $n$, and $A\subset_{\mathrm{cb},\gamma}B$ when there exists $\gamma'<\gamma$ with $A\subseteq_{\mathrm{cb},\gamma}B$.
\end{definition}

\section{Very rapidly increasing sequences and the Cuntz semigroup}\label{Sect3}

We start by noting that, for close $C^*$-algebras of real rank zero, an isomorphism between their Cuntz semigroups can be deduced from existing results in the literature. For a $C^*$-algebra $A$, let $V(A)$ be the Murray and von Neumann semigroup of equivalence classes of projections in $\bigcup_{n=1}^\infty A\otimes M_n$ and write $\Sigma(V(A))=\{[p]\in V(A)\mid p=p^2=p^*\in A\}$. This is a local semigroup in the sense that if $p$, $q$, $p'$ and $q'$ are projections in $A$ with $p'q'=0$ and $p\sim p'$, $q\sim q'$, then $[p]+[q]=[p'+q']\in\Sigma(V(A))$. Recall that, if $A$ has real rank zero, then the work of Zhang \cite{Z:PJM} shows that $V(A)$ has the Riesz refinement property. By definition, this means that whenever $x_1,\dots,x_n,y_1,\dots,y_m\in V(A)$ satisfy $\sum_ix_i=\sum_jy_j$, then there exist $z_{i,j}\in V(A)$ with $\sum_jz_{i,j}=x_i$ and $\sum_iz_{i,j}=y_j$ for each $i,j$. The case $m=n=2$ of this can be found as \cite[Lemma 2.3]{AP:PAMS}, and the same proof works in general.

The Cuntz semigroup of a $C^*$-algebra of real rank zero is completely determined by its semigroup of projections (see \cite{Pe:IJM} when $A$ additionally has stable rank one and \cite{ABP:IJM} for the general case). We briefly recall how this is done. An \emph{interval} in $V(A)$ is a non-empty, order hereditary and upward directed subset $I$ of $V(A)$, which is said to be countably generated provided there is an increasing sequence $(x_n)$ in $V(A)$ such that $I=\{x\in V(A)\mid x\leq x_n\text{ for some }n\}$. The set of countably generated intervals is denoted by $\Lambda_{\sigma}(V(A))$, and it has a natural semigroup structure. Namely, if $I$ and $J$ have generating sequences $(x_n)$ and $(y_n)$ respectively, then $I+J$ is the interval generated by $(x_n+y_n)$. Given a positive element $a$ in $A\otimes\K$ in a $\sigma$-unital $C^*$-algebra of real rank zero $A$, put $I(a)=\{[p]\in V(A)\mid p\precsim a\}$. The correspondence $[a]\mapsto I(a)$ defines an ordered semigroup isomorphism $\Cu(A)\cong\Lambda_{\sigma}(V(A))$. 
\begin{theorem}
\label{thm:rr0} Let $A$ and $B$ be $\sigma$-unital $C^*$-subalgebras of a $C^*$-algebra $C$, with $d(A,B)<1/8$. If $A$ has real rank zero, then $B$ also has real rank zero and $\Cu(A)\cong\Cu(B)$.
\end{theorem}
\begin{proof}
That $B$ has real rank zero follows from \cite[Theorem 6.3]{CSSW:GAFA}. We know from \cite[Theorem 2.6]{PR:CJM} that there is an isomorphism of local semigroups $\Phi_1\colon\Sigma(V(A))\to \Sigma(V(B))$ (with inverse, say, $\Psi_1$). This is defined as $\Phi_1[p]=[q]$, where $q$ is a projection in $B$ such that $\Vert p-q\Vert<1/8$. Given $p\in M_n(A)$, by \cite[Theorem 3.2]{Z:PJM} we can find projections $\{p_i\}_{i=1,\ldots,n}$ in $A$ such that $[p]=\sum_i [p_i]$. Now extend $\Phi_1$ to $\Phi\colon V(A)\to V(B)$ by $\Phi([p])=\sum_i \Phi_1([p_i])$. Let us check that $\Phi$ is well defined. If $[p]=\sum_i[p_i]=\sum_j[q_j]$ for projections $p_i$ and $q_j$ in $A$, then use refinement to find elements $a_{ij}\in V(A)$ such that $[p_i]=\sum_ja_{ij}$ and $[q_j]=\sum_ia_{ij}$. We may also clearly choose projections $z_{ij},z_{ij}'\in A$ such that $a_{ij}=[z_{ij}]=[z_{ij}']$, and such that $z_{ij}\perp z_{ik}$ if $j\neq k$, and $z_{ij}'\perp z_{lj}'$ if $i\neq l$. Then:
\begin{align*}
\sum\Phi_1([p_i])&=\sum_i\sum_j\Phi_1([z_{ij}])\\
&=\sum_i\sum_j\Phi_1([z_{ij}'])=\sum_j\sum_i\Phi_1([z_{ij}'])=\sum_j\Phi_1([q_j])\,.
\end{align*}
It is clear that $\Phi$ is additive and that $\Phi_{|\Sigma(V(A))}=\Phi_1$. Using $\Psi_1$, we construct an additive map $\Psi\colon V(B)\to V(A)$, with $\Psi_{|\Sigma(V(B))}=\Psi_1$.  Since $\Psi_1\circ\Phi_1=\id_{\Sc(V(A))}$, it follows that $\Psi\circ\Phi=\id_{V(A)}$. Similarly $\Phi\circ\Psi=\id_{V(B)}$.

Now, since $\Cu(A)\cong\Lambda_{\sigma}(V(A))$ and $\Cu(B)\cong\Lambda_{\sigma}(V(B))$, it follows that $\Cu(A)$ is isomorphic to $\Cu(B)$.
\end{proof}

We turn now to very rapidly increasing sequences. These provide the key tool we use to transfer information between close algebras at the level of the Cuntz semigroup.

\begin{definition}
Let $A$ be a $C^*$-algebra. We say that a rapidly increasing sequence $(a_n)_{n=1}^\infty$ in $A^+_1$ is \emph{very rapidly increasing} if given $\eps>0$ and $n\in\mathbb N$, there exists $m_0\in\mathbb N$ such that for $m\geq m_0$, there exists $v\in A_1$ with $\|(va_mv^*)a_n-a_n\|<\eps$.  Say that a very rapidly increasing sequence $(a_n)_{n=1}^\infty$ in $(A\otimes\K)_+^1$ \emph{represents} $x\in\Cu(A)$ if $\sup_n\langle a_n\rangle =x$.
\end{definition}

The following two functions are used in the sequel to manipulate very rapidly increasing sequences.  Given $a\in A_+$ and $\eps>0$, write $(a-\eps)_+$ for $h_\eps(a)$, where $h_\eps$ is the continuous function $h_\eps(t)=\max(0,t-\eps)$. For $0\leq\beta<\gamma$, let $g_{\beta,\gamma}$ be the piecewise linear function on $\mathbb R$ given by 
\begin{equation}\label{G}
g_{\beta,\gamma}(t)=\begin{cases}0,&t\leq \beta;\\\frac{t-\beta}{\gamma-\beta},&\beta<t<\gamma;\\1,&t\geq \gamma.\end{cases}
\end{equation}

With this notation, the standard example of a very rapidly increasing sequence is given by $(g_{2^{-(n+1)},2^{-n}}(a))_{n=1}^\infty$ for $a\in A^+_1$. This sequence represents $\langle a\rangle$.  In this way every element of the Cuntz semigroup of $A$ is represented by a very rapidly increasing sequence from $(A\otimes\K)^+_1$.  In the next few lemmas we develop properties of very rapidly increasing sequences, starting with a technical observation.

\begin{lemma}\label{Key}
 Let $A$ be a $C^*$-algebra and let $a,b\in A_1^+$ and $v\in A_1$ satisfy $\|v^*bva-a\|\leq\delta$ for some $\delta>0$. Suppose that $0<\beta<1$ and $\gamma\geq 0$ satisfy $\gamma+\delta\beta^{-1}<1$, then
$\langle (a-\beta)_+\rangle \leq\langle (b-\gamma)_+\rangle$ in $\Cu(A)$. 
\end{lemma}
\begin{proof}
Let $p\in A^{**}$ denote the spectral projection of $a$ for the interval $[\beta,1]$. When $p=0$, then $(a-\beta)_+=0$ and the result is trivial, so we may assume that $p\neq 0$. Then $ap$ is invertible in $pA^{**}p$ with inverse $x$ satisfying $\|x\|\leq \beta^{-1}$.  Compressing $(v^*bva-a)$ by $p$ and multiplying by $x$, we have $\|pv^*bvp-p\|\leq\delta\beta^{-1}$. Thus
$$
\|pv^*(b-\gamma)_+vp-p\|\leq\|(b-\gamma)_+-b\|+\|pv^*bvp-p\|\leq\gamma+\delta\beta^{-1},
$$
and so
$$
pv^*(b-\gamma)_+vp\geq \left(1-(\gamma+\delta\beta^{-1})\right)p.
$$
As $p$ acts as a unit on $(a-\beta)_+$, we have
\begin{align*}
(a-\beta)_+&=(a-\beta)_+^{1/2}p(a-\beta)_+^{1/2}\\
&\leq \left(1-(\gamma+\delta\beta^{-1})\right)^{-1}(a-\beta)_+^{1/2}pv^*(b-\gamma)_+vp(a-\beta)_+^{1/2}\\
&=\left(1-(\gamma+\delta\beta^{-1})\right)^{-1}(a-\beta)_+^{1/2}v^*(b-\gamma)_+v(a-\beta)_+^{1/2}.
\end{align*}
Thus $(a-\beta)_+\precsim (b-\gamma)_+$. 
\end{proof}

The next lemma encapsulates the fact that the element of the Cuntz semigroup represented by a very rapidly increasing sequence $(a_n)_{n=1}^\infty$ of contractions depends only on the behaviour of parts of the $a_n$ with spectrum near $1$.
 
 \begin{lemma}\label{L:VRI3}
Let $(a_n)_{n=1}^\infty$ be a very rapidly increasing sequence in $A^+_1$. Then for each $\lambda<1$, the sequence $(\langle (a_n-\lambda)_+\rangle)_{n=1}^\infty$ has the property that for each $n\in\mathbb N$, there is $m_0\in\mathbb N$ such that for $m\geq m_0$, we have $\langle (a_n-\lambda)_+\rangle\ll\langle (a_m-\lambda)_+\rangle$. Furthermore
\begin{equation}\label{E1}
\sup_n\langle (a_n-\lambda)_+\rangle=\sup_n\langle a_n\rangle.
\end{equation}
\end{lemma}
\begin{proof}
Fix $n\in\mathbb N$ and $0<\eps<\lambda$  and take $0<\delta$ small enough that $\lambda+\eps^{-1}\delta<1$. As $(a_n)_{n=1}^\infty$ is very rapidly increasing, there exists $m_0$ such that for $m\geq m_0$, there exists $v\in A_1$ with $\|(v^*a_mv)a_n-a_n\|<\delta$.   Lemma \ref{Key} gives
$$
\langle(a_n-\eps)_+\rangle\leq\langle (a_m-\lambda)_+\rangle,
$$
so that $\langle (a_n-\lambda)_+\rangle\ll\langle a_m-\lambda)_+\rangle$ as $\eps<\lambda$. This shows that $(\langle (a_r-\lambda)_+\rangle)_{r=1}^\infty$ is upward directed and that
$$
\langle (a_n-\delta)_+\rangle\leq \sup_r\langle(a_r-\lambda)_+\rangle,
$$
for all $n$ and all $\eps>0$, from which (\ref{E1}) follows.
\end{proof}

We can modify elements sufficiently far down a very rapidly increasing sequences with contractions so that they almost act as units for positive contractions dominated in the Cuntz semigroup by the sequence.

\begin{lemma}\label{L:VRI1}
Let $A$ be a $C^*$-algebra.
\begin{enumerate}[(1)]
\item  Suppose that $a,b\in A_1^+$ satisfy $a\precsim b$. Then for all $\eps>0$, there exists $v\in A$ with $\|v^*bva-a\|\leq\eps$ and $\|v^*bv\|\leq 1$.\label{L:VRI1:1}
\item Let $(a_n)_{n=1}^\infty$ be a very rapidly increasing sequence in $A^+_1$ and suppose $a\in A_1^+$ satisfies $\langle a\rangle \ll\sup\langle a_n\rangle$. Then, for every $\eps>0$, there exists $m_0\in\mathbb N$ such that for $m\geq m_0$, there exists $v\in A_1$ with $\|(v^*a_mv)a-a\|<\eps$.\label{L:VRI1:2}
\end{enumerate}
\end{lemma}
\begin{proof}
(\ref{L:VRI1:1}). Fix $\eps>0$ and find $r>0$ so that $\|a^{1+r}-a\|\leq\eps/2$.  Now $a^r\precsim b$, so there exists $w\in A$ with $\|a^r-w^*bw\|\leq\eps/4$. Thus $\|w^*bw\|\leq 1+\eps/4$, and so, writing $v=(1+\eps/4)^{-1/2}w$, we have $\|v^*bv\|\leq 1$ and $\|w^*bw-v^*bv\|\leq \eps/4$. As such $\|a^r-v^*bv\|\leq \eps/2$ and so 
$$
\|v^*bva-a\|\leq\|v^*bv-a^r\|\|a\|+\|a^{1+r}-a\|\leq\eps/2+\eps/2=\eps,
$$
as claimed.

(\ref{L:VRI1:2}). As $\langle a\rangle\ll \sup\langle a_n\rangle$, there exists some $m_1\in\mathbb N$ with $a\precsim a_{m_1}\sim a^2_{m_1}$.  Fix $\eps>0$ and by part (\ref{L:VRI1:1}), find $w\in A$ with
 $\|(w^*a_{m_1}^2w)a-a\|<\eps/2$ and $\|w^*a_{m_1}^2w\|\leq 1$.  Now set $\eps'=\eps/(2\|w\|)$ and, as $(a_n)_{n=1}^\infty$ is very rapidly increasing, find some $m_0>m_1$ such that for $m\geq m_0$ there exists  $t\in A_1$ with $\|(t^*a_mt)a_{m_1}-a_{m_1}\|\leq\eps'$.  Given such $m$ and $t$, we have
\begin{align*}
\|(w^*a_{m_1}t^*a_mta_{m_1}w)a-a\|&\leq \|w^*a_{m_1}\|\|(t^*a_mt)a_{m_1}-a_{m_1}\|\|w\|\|a\|\\
&\quad +\|(w^*a_{m_1}^2w)a-a\|\\
&\leq \|w\|\eps'+\eps/2=\eps,
\end{align*}
as $\|w^*a_{m_1}\|\leq 1$.  As such we can take $v=ta_{m_1}w\in A_1$.
\end{proof}

It follows immediately from part (\ref{L:VRI1:2}) above, that two very rapidly increasing sequences representing the same element of the Cuntz semigroup can be intertwined to a single very rapidly increasing sequence.
\begin{proposition}\label{L:VRI2}
Let $(a_n)_{n=1}^\infty,(a_n')_{n=1}^\infty$ be very rapidly increasing sequences in a $C^*$-algebra $A$ representing the same element $x\in\Cu(A)$. Then these sequences can be intertwined after telescoping to form a very rapidly increasing sequence which also represents $x$, i.e. there exists $m_1<m_2<\cdots$ and $n_1<n_2<\cdots$ such that $(a_{m_1},a'_{n_1},a_{m_2},a'_{n_2},\cdots)$ is a very rapidly increasing sequence.
\end{proposition}

Given a rapidly increasing sequence in $A^+_1$, we can use the functions $g_{\beta,\gamma}$ from (\ref{G}) to push the spectrum of the elements of the sequence out to $1$ and extract a very rapidly rapidly increasing sequence representing the same element of the Cuntz semigroup.
\begin{lemma}\label{L:VRI4}
Let $A$ be a $C^*$-algebra and $(a_n)_{n=1}^\infty$ be a rapidly increasing sequence in $A^+_1$. Then there exists a sequence $(m_n)_{n=1}^\infty$ in $\mathbb N$ such that the sequence $(g_{2^{-(m_n+1)},2^{-m_n}}(a_n))_{n=1}^\infty$ is very rapidly increasing and $$\sup_n\langle g_{2^{-(m_n+1)},2^{-m_n}}(a_n)\rangle=\sup_n\langle a_n\rangle.$$  In particular, every element of the scale $\Sc(\Cu(A))$ can be expressed as a very rapidly increasing sequence of elements from $A_1^+$.
\end{lemma}
\begin{proof}
We will construct the $m_n$ so that $a_{n-1}\precsim g_{2^{-(m_n+1)},2^{-m_n}}(a_n)$ and for each $1\leq r<n$, there exists $v\in A_1$ with $$
\|(v^*g_{2^{-(m_n+1)},2^{-m_n}}(a_n)v)g_{2^{-(m_r+1)},2^{-m_r}}(a_r)-g_{2^{-(m_r+1)},2^{-m_r}}(a_r)\|<2^{-n}.
$$
Fix $n\in\mathbb N$ and suppose $m_1,\cdots,m_{n-1}$ have been constructed with these properties.  As $(g_{2^{-(m+1)},2^{-m}}(a_n))_{m=1}^\infty$ is a very rapidly increasing sequence representing $\langle a_n\rangle$ and $\langle a_{n-1}\rangle\ll\langle a_n\rangle$, there exists $\widetilde{m_n}$ such that $\langle a_{n-1}\rangle \ll\langle (g_{2^{-(m+1)},2^{-m}}(a_n))\rangle$ for $m\geq \widetilde{m_n}$.  Further, for  $1\leq r<n$, $$\langle g_{2^{-(m_r+1)},2^{-m_r}}(a_r)\rangle\ll\langle a_r\rangle\ll\sup_m\langle (g_{2^{-(m+1)},2^{-m}}(a_n)\rangle$$ and so the required $m_n$ can be found using Part (2) of Lemma \ref{L:VRI1}.  

 The resulting sequence $(g_{2^{-(m_n+1)},2^{-m_n}}(a_n))_{n=1}^\infty$ is very rapidly increasing by construction. As $a_{n-1}\precsim g_{2^{-(m_n+1)},2^{-m_n}}(a_n)\precsim a_n$ for all $n$, we have $\sup_n\langle g_{2^{-(m_n+1)},2^{-m_n}}(a_n)\rangle=\sup_n\langle a_n\rangle$.
\end{proof}

We now consider the situation where we have two close $C^*$-algebras acting on the same Hilbert space.  The following lemma ensures that we can produce a well defined map between the Cuntz semigroups.
\begin{lemma}\label{Lem:Claim}
Let $A,B$ be $C^*$-algebras acting on the same Hilbert space and suppose that $a\in A_1^+$ and $b\in B_1^+$ satisfy $\|a-b\|<2\alpha$ for some $\alpha<1/27$. Suppose that $(a_n)_{n=1}^\infty$ is a very rapidly increasing sequence in $A_1^+$ with $\langle a\rangle\ll\sup\langle a_n\rangle$. Then,  there exists $n_0\in\mathbb N$ with the property that for $n\geq n_0$ and $b_n\in B^+_1$ with $\|b_n-a_n\|<2\alpha$, we have
$$
\langle (b-18\alpha)_+\rangle\ll\langle (b_n-\gamma)_+\rangle\ll\langle (b_n-18\alpha)_+\rangle
$$
in $\Cu(B)$, for all $\gamma$ with $18\alpha<\gamma<2/3$.
\end{lemma}
\begin{proof}
Fix $\gamma$ with $2/3>\gamma>18\alpha$ (which is possible as $\alpha<1/27$). By taking $\eps=2\alpha-\|a-b\|$ in Lemma \ref{L:VRI1} (\ref{L:VRI1:2}), there exists $n_0\in\mathbb N$ such that for $n\geq n_0$, there exists $v\in A_1$ with $$
\|(v^*a_nv)a-a\|<2\alpha-\|a-b\|.
$$
Fix such an $n\geq n_0$ and $v\in A_1$, and take $b_n\in B^+_1$ with $\|a_n-b_n\|<2\alpha$ and choose some $w\in B_1$ with $\|w-v\|<\alpha$. We have
$$
\|w^*b_nw-v^*a_nv\|\leq 2\|w-v\|+\|b_n-a_n\|<4\alpha
$$
so that
\begin{align*}
\|(w^*b_nw)b-b\|&\leq\|((w^*b_nw)-1)(b-a)\|\\
&\quad+\|(w^*b_nw-v^*a_nv)a\|\\
&\quad+\|(v^*a_nv)a-a\|\\
&\leq \|b-a\|+4\alpha+\|(v^*a_nv)a-a\|\\
&\leq 6\alpha.
\end{align*}
Taking $\delta=6\alpha$, $\beta=18\alpha$ and $2/3>\gamma'>\gamma>18\alpha$ so that $\gamma'+\delta\beta^{-1}<1$, Lemma \ref{Key} gives $$\langle (b-18\alpha)_+\rangle\leq\langle (b_n-\gamma')_+\rangle\ll\langle (b_n-\gamma)_+\rangle\ll\langle (b_n-18\alpha)_+\rangle,$$ as claimed.
\end{proof}

\begin{proposition}\label{CuMap1}
Let $A,B$ be $C^*$-algebras acting on the same Hilbert space with the property that there exists $\alpha<1/27$ such that for each $a\in A_1$ there exists  $b\in B_1$ with $\|a-b\|<\alpha$.  Then there is a well defined, order preserving map $\Phi:\Sc(\Cu(A))\rightarrow\Sc(\Cu(B))$ given by 
$$
\Phi(\sup\langle a_n\rangle)=\sup\langle (b_n-18\alpha)_+\rangle,
$$ whenever $(a_n)_{n=1}^\infty$ is a very rapidly increasing sequence in $A^+_1$ and $b_n\in B^+_1$ have $\|a_n-b_n\|<2\alpha$ for all $n\in\mathbb N$. Moreover, if $d(A,B)<\alpha$ for $\alpha<1/42$, then $\Phi$ is a bijection with inverse  $\Psi:\Sc(\Cu(B))\rightarrow \Sc(\Cu(A))$ obtained from interchanging the roles of $A$ and $B$ in the definition of $\Phi$.
\end{proposition}
\begin{proof}
Suppose first that $\alpha<1/27$. To see that $\Phi$ is well defined, we apply Lemma \ref{Lem:Claim} repeatedly.  Firstly, given a very rapidly increasing sequence $(a_n)_{n=1}^\infty$ in $A_1^+$ representing an element $x\in\Sc(\Cu(A))$ and a sequence $(b_n)_{n=1}^\infty$ in $B^+_1$ with $\|a_n-b_n\|<2\alpha$ for all $n$, Lemma \ref{Lem:Claim} shows that the sequence $(\langle (b_n-18\alpha)_+\rangle)_{n=1}^\infty$ is upward directed. Indeed, for each $m$, take $a=a_m$ and $b=b_m$ in Lemma \ref{Lem:Claim} so that $\langle (b_m-18\alpha)_+\rangle\ll\langle (b_n-18\alpha)_+\rangle$ for all sufficiently large $n$. As such $\sup_n\langle(b_n-18\alpha)_+\rangle$ exists in $\Sc(\Cu(B))$. 

Secondly, this supremum does not depend on the choice of $(b_n)_{n=1}^\infty$. Consider two sequences $(b_n)_{n=1}^\infty$ and $(b_n')_{n=1}^\infty$ satisfying $\|b_n-a_n\|<2\alpha$ and $\|b_n'-a_n\|<2\alpha$ for all $n$. For each $n$, Lemma \ref{Lem:Claim} shows that there exists $m_0$ such that for $m\geq m_0$, we have $$
\langle (b_n-18\alpha)_+\rangle\ll\langle (b_m'-18\alpha)_+\rangle,\quad\text{and}\quad \langle (b_n'-18\alpha)_+\rangle\ll\langle (b_m-18\alpha)_+\rangle.
$$
Thus $\sup_n\langle(b_n-18\alpha)_+\rangle=\sup_n\langle (b_n'-18\alpha)_+\rangle$.  

Thirdly, given two very rapidly increasing sequences $(a_n')_{n=1}^\infty$ and $(a_n)_{n=1}^\infty$ in $A^+_1$ with $\sup_n\langle a_n'\rangle\leq\sup_n\langle a_n\rangle$, and sequences $(b_n')_{n=1}^\infty$ and $(b_n)_{n=1}^\infty$ in $B^+_1$ with $\|b_n'-a_n'\|,\|b_n-a_n\|<2\alpha$ for all $n$, Lemma \ref{Lem:Claim} gives $\sup_n\langle (b_n'-18\alpha)_+\rangle\leq\sup_n\langle (b_n-18\alpha)_+\rangle$. In particular, when $(a_n')_{n=1}^\infty$ and $(a_n)_{n=1}^\infty$ represent the same element of $\Sc(\Cu(A))$, this shows that the map $\Phi$ given in the proposition is well defined.  In general, this third observation also shows that $\Phi$ is order preserving.

Now suppose that $d(A,B)<\alpha<1/42$ and let $\Psi:\Sc(\Cu(B))\rightarrow\Sc(\Cu(A))$ be the order preserving map obtained by interchanging the roles of $A$ and $B$ above.   Take $x\in\Sc(\Cu(A))$ and fix a very rapidly increasing sequence $(a_n)_{n=1}^\infty$ in $A^+_1$ representing $x$. Fix a sequence $(b_n)_{n=1}^\infty$ in $B^+_1$ with $\|a_n-b_n\|<2\alpha$ for all $n$.  For each $n$, Lemma \ref{Lem:Claim} gives $m>n$ with 
$$
\langle(b_n-18\alpha)_+\rangle\ll\langle (b_m-\gamma)_+\rangle\ll\langle (b_m-18\alpha)_+\rangle,
$$
for any $\gamma$ with $18\alpha<\gamma<2/3$. Passing to a subsequence if necessary, we can assume this holds for $m=n+1$ and hence $((b_n-18\alpha)_+)_{n=1}^\infty$ is a rapidly increasing sequence.  By Lemma \ref{L:VRI4}, there exists a sequence $(m_n)_{n=1}^\infty$ in $\mathbb N$ so that, defining $b_n'=g_{2^{-(m_n+1)},2^{-m_n}}((b_n-18\alpha)_+)$, we have a very rapidly increasing sequence $(b_n')_{n=1}^\infty$ in $B^+_1$ with
$$
\sup_n\langle b_n'\rangle=\sup_n\langle (b_n-18\alpha)_+\rangle=\Phi(x).
$$
Choose a sequence $(c_n)_{n=1}^\infty$ in $A^+_1$ with $\|c_n-b_n'\|<2\alpha$ for each $n$ so that the definition of $\Psi$ gives $\Psi(\Phi(x))=\sup\langle (c_n-18\alpha)_+\rangle$.  We now show that $x\leq\Psi(\Phi(x))\leq x$.

Fix $0<\beta<1$ with $\alpha(18+24\beta^{-1})<1$. This choice can be made as $\alpha<1/42$. Fix $n\in\mathbb N$. As $\langle (b_n-18\alpha)_+\rangle \ll\sup_r\langle b_r'\rangle$, Lemma \ref{L:VRI1} (\ref{L:VRI1:2}) provides $m_0\in\mathbb N$ such that for $m\geq m_0$, there exists $w\in B_1$ with 
\begin{equation}\label{P3.8:1}
\|(w^*b_m'w)(b_n-18\alpha)_+-(b_n-18\alpha)_+\|<2\alpha-\|a_n-b_n\|.
\end{equation}
Take $v\in A_1$ with $\|v-w\|<\alpha$.  Then
\begin{equation}\label{P3.8:2}
\|(v^*c_mv-w^*b_m'w)\|\leq \|c_m-b_m'\|+2\|v-w\|<4\alpha.
\end{equation}
Combining the estimates (\ref{P3.8:1}), (\ref{P3.8:2}) and noting that $\|w^*b_m'w-1\|\leq 1$ as $w$ is a contraction, gives
\begin{align*}
\|(v^*c_mv)a_n-a_n\|&\leq \|((v^*c_mv)-1)(a_n-b_n)\|\\
&\quad+\|(v^*c_mv-w^*b_m'w)b_n\|\\
&\quad+\|(w^*b_m'w-1)(b_n-(b_n-18\alpha)_+)\|\\
&\quad+\|(w^*b_m'w)(b_n-18\alpha)_+-(b_n-18\alpha)_+\|\\
&<\|a_n-b_n\|+4\alpha+18\alpha+(2\alpha-\|a_n-b_n\|)=24\alpha.
\end{align*}
Taking $\gamma=18\alpha$, $\delta=24\alpha$, Lemma \ref{Key} gives
$$
\langle (a_n-\beta)_+\rangle\leq\langle(c_m-18\alpha)_+\rangle\leq\Psi(\Phi(x)).
$$
As $n$ was arbitrary, $\sup_n\langle (a_n-\beta)_+\rangle\leq\Psi(\Phi(x))$.   As $\beta<1$, Lemma \ref{L:VRI3} gives $\sup_n\langle (a_n-\beta)_+\rangle=\sup_n\langle a_n\rangle=x$ so that $x\leq\Psi(\Phi(x))$.

For the reverse inequality, fix $k\in\mathbb N$ and apply Lemma \ref{Lem:Claim} (with the roles of $A$ and $B$ reversed, $b_k'$ playing the role of $a$, $(b_n')_{n=1}^\infty$ the role of $(a_n)$) and $\gamma=1/2$, so $18\alpha<\gamma<2/3$) to find some $n\in\mathbb N$ such that $\langle (c_k-18\alpha)_+\rangle\leq\langle (c_n-1/2)_+\rangle$.  Now, just as in the proof of Lemma \ref{Lem:Claim}, there is $z\in B_1$ with $\|(z^*b_{n+1}z)b_n-b_n\|\leq6\alpha$. Let $p\in B^{**}$ be the spectral projection of $b_n$ for $[18\alpha,1]$, so that, just as in the proof of Lemma \ref{Key}, $\|z^*b_{n+1}zp-p\|\leq 1/3$.  Fix $y\in A_1$ with $\|y-z\|\leq\alpha$. Since $p$ is a unit for $(b_n-18\alpha)_+$, it is a unit for $b_n'=g_{2^{-(m_n+1)},2^{-m_n}}((b_n-18\alpha)_+)$, giving the estimate
\begin{align*}
\|y^*a_{n+1}yc_n-c_n\|&\leq\|y^*a_{n+1}yc_n-z^*b_{n+1}zc_n\|\\
&\quad+\|(z^*b_{n+1}z-1)(c_n-b_n')\|\\
&\quad+\|(z^*b_{n+1}z)b_n'-b_n'\|\\
&\leq4\alpha+2\alpha+1/3=6\alpha+1/3.
\end{align*}
Take $\delta=6\alpha+1/3$, $\beta=1/2$ and $\gamma=0$, so that $\gamma+\beta^{-1}\delta=2/3+12\alpha<1$. Thus Lemma \ref{Key}, gives
$$
\langle(c_n-1/2)_+\rangle\leq\langle a_{n+1}\rangle,
$$
and hence
$$
\langle (c_k-18\alpha)_+\rangle\leq\langle a_{n+1}\rangle\leq x.
$$
Taking the supremum over $k$ gives $\Psi(\Phi(x))\leq x$.
\end{proof}

\begin{theorem}\label{CuMap2}
Let $A$ and $B$ be $C^*$-algebras acting on the same Hilbert space with $d_{\cb}(A,B)<\alpha<1/42$. Then $(\Cu(A),\Sc(\Cu(A)))$ is isomorphic to $(\Cu(B),\Sc(\Cu(B)))$.  Moreover, an order preserving isomorphism $\Phi:\Cu(A)\rightarrow \Cu(B)$ can be defined by $\Phi(\sup\langle a_n\rangle)=\sup\langle (b_n-18\alpha)_+\rangle$, whenever $(a_n)_{n=1}^\infty$ is a very rapidly increasing sequence in $(A\otimes\mathcal K)^+_1$ and $(b_n)_{n=1}^\infty$ is a sequence in $(B\otimes\mathcal K)_1^+$ with $\|a_n-b_n\|<2\alpha$ for all $n\in\mathbb N$.

\end{theorem}
\begin{proof}
We have $d(A\otimes \mathcal K,B\otimes\mathcal K)<\alpha<1/42$.  By definition $\Sc(\Cu(A\otimes\K))=\Cu(A)$ and $\Sc(\Cu(A\otimes\K))=\Cu(B)$. By applying Proposition \ref{CuMap1} to $A\otimes\mathcal K$ and $B\otimes \mathcal K$, we obtain mutually inverse order preserving bijections $\Phi:\Cu(A\otimes \mathcal K)\rightarrow\Cu(B\otimes \mathcal K)$ and $\Psi:\Cu(B\otimes \mathcal K)\rightarrow\Cu(A\otimes \mathcal K)$, given by $\Phi(\sup\langle a_n\rangle)=\sup\langle (b_n-18\alpha)_+\rangle$, whenever $(a_n)_{n=1}^\infty$ is a very rapidly increasing sequence in $(A\otimes\mathcal K)^+_1$ and $(b_n)_{n=1}^\infty$ is a sequence in $(B\otimes\mathcal K)_1^+$ with $\|a_n-b_n\|<2\alpha$ for all $n\in\mathbb N$.  Given a very rapidly increasing sequence $(a_n)_{n=1}^\infty$ in $A^+_1$ representing an element $x\in\Sc(\Cu(A))$, we can find a sequence $(b_n)_{n=1}^\infty$ in $B^+_1$ with $\|a_n-b_n\|<2\alpha$, so that $\Phi(x)=\sup\langle (b_n-18\alpha)_+\rangle\in\Sc(\Cu(B))$.  Since $\Phi$ and $\Phi^{-1}$ are order preserving bijections, they also preserve the relation $\ll$ of compact containment and suprema of countable upward directed sets, as these notions are determined by the order relation $\leq$.  Further, taking $a_n=b_n=0$ for all $n$, shows that $\Phi(0_{\Cu(A)})=0_{\Cu(B)}$. Finally, note that $\Phi$ preserves addition: given very rapidly increasing sequences $(a_n)_{n=1}^\infty$ and $(a_n')_{n=1}^\infty$ in $(A\otimes\mathcal K)_1^+$ representing $x$ and $y$ in $\Cu(A)$, the sequence $(a_n\oplus a_n')$ is very rapidly increasing in $M_2(A\otimes\mathcal K)\cong A\otimes\mathcal K$. If $(b_n)_{n=1}^\infty,(b_n')_{n=1}^\infty$ have $\|a_n-b_n\|,\|a_n'-b_n'\|<2\alpha$ for all $n$, then 
$$
\|(a_n\oplus a_n')-(b_n\oplus b_n')\|<2\alpha,
$$
and has
$$
((b_n\oplus b_n')-18\alpha)_+=(b_n-18\alpha)_+\oplus(b_n'-18\alpha)_+.
$$ In this way we see that $\Phi(x+y)=\Phi(x)+\Phi(y)$.
\end{proof}

In particular properties of a $C^*$-algebra which are determined by its Cuntz semigroup transfer to completely close $C^*$-algebras. One of the most notable of these properties is that of strict comparison.
\begin{corollary}
Let $A$ and $B$ be $C^*$-algebras acting on the same Hilbert space with $d_\cb(A,B)<1/42$ and suppose that $A$ has strict comparison. Then so too does $B$.
\end{corollary}
\section{$\Z$-stability and automatic complete closeness}\label{Sect4}

Given a $C^*$-algebra $A\subset\Bb(\Hs)$, \cite{CCSSWW:arXiv} shows that the metrics $d(A,\cdot)$ and $d_\cb(A,\cdot)$ are equivalent if and only if $A$ has a positive answer to Kadison's similarity problem from \cite{K:AJM}.  The most useful reformulation of the similarity property for working with close $C^*$-algebras is due to Christensen and Kirchberg. Combining \cite{K:JOT} and \cite[Theorem 3.1]{C:Scand}, it follows that a $C^*$-algebra $A$ has a positive answer to the similarity problem if and only if $A$ has Christensen's property $D_k$ from \cite{C:Acta} for some $k$.
 
\begin{definition}\label{D:DK}
Given an operator $T\in\Bb(\Hs)$, we write $\ad(T)$ for the derivation $\ad(T)(x)=xT-Tx$.  A $C^*$-algebra $A$ has \emph{property $D_k$} for some $k>0$ if, for every non-degenerate representation $\pi:A\rightarrow\Bb(\Hs)$, the inequality
\begin{equation}\label{DefDK}
d(T,\pi(A)')\leq k\|\ad(T)|_{\pi(A)}\|
\end{equation}
holds for all $T\in\Bb(\Hs)$.  A von Neumann algebra $A$ is said to have the property $D_k^*$ if the inequality (\ref{DefDK}) holds for all unital normal representations $\pi$ on $\Hs$ and all $T\in\Bb(\Hs)$.  
\end{definition}

By taking weak$^*$-limit points, it follows that if $A$ is a weak$^*$-dense $C^*$-subalgebra of a von Neumann algebra $\M$ and $A$ has property $D_k$, then $\M$ has property $D_k^*$. 

That property $D_k$ converts near containments to completely bounded near containments originates in \cite[Theorem 3.1]{C:Acta}. The version we give below improves on the bounds $\gamma'=6k\gamma$ from \cite{C:Acta} and $\gamma'=(1+\gamma)^{2k}-1$ from \cite[Corollary 2.12]{CSSW:GAFA}. 

\begin{proposition}\label{DKCB}
Suppose that $A$ has property $D_k$ for some $k>0$. Then for $\gamma>0$, every near inclusion $A\subseteq_{\gamma}B$ (or $A\subset_\gamma B)$ with $A$ and $B$ acting non-degenerately on the same Hilbert space, gives rise to a completely bounded near inclusion $A\subseteq_{\cb,\gamma'}B$ (or $A\subset_{\cb,\gamma'}$), where $\gamma'=2k\gamma$.
\end{proposition}
\begin{proof}
Suppose $A\subseteq_\gamma B$ is a near inclusion of $C^*$-algebras acting non-degenerately on $\Hs$ and fix $n\in\mathbb N$. Let $C=C^*(A,B)$ and let $\pi:C\rightarrow\Bb(\Ks)$ be the universal representation of $C$. Then $\pi(A)''$ has property $D_k^*$ so that $\pi(A)''\subseteq_{\cb,2k\gamma}\pi(B)''$ by \cite[Proposition 2.2.4]{CCSSWW:Preprint}.  By definition, for $n\in\mathbb N$ we have $\pi(A)''\otimes M_n\subseteq_{2k\gamma}\pi(B)''\otimes M_n$. As $\pi$ is the universal representation of $C$ the Hahn-Banach argument used to deduce \cite[equation (3)]{C:Acta} from \cite[equation (2)]{C:Acta} gives $A\otimes M_n\subseteq_{2k\gamma}B\otimes M_n$, as required.  The result when we work with strict near inclusions $A\subset_\gamma B$ follows immediately.
\end{proof}

$C^*$-algebras with no bounded traces (such as stable algebras) where shown to have the similarity property in \cite{H:Ann}. Using the property $D_k$ version of this fact, the previous proposition gives automatic complete closeness when one algebra has no bounded traces. The argument below which transfers the absence of bounded traces to a nearby $C^*$-algebra essentially goes back to \cite[Lemma 9]{KK:AJM}. We use more recent results in order to get better estimates.
\begin{corollary}\label{NoTraces}
Suppose that $A$ and $B$ are $C^*$-algebras which act non-degenerately on the same Hilbert space and satisfy $d(A,B)<\gamma$ for $\gamma<(2+6\sqrt{2})^{-1}$. Suppose that $A$ has no bounded traces (for example if $A$ is stable).  Then $B$ has no bounded traces, and therefore $A\subset_{\cb,3\gamma}B$, $B\subset_{\cb,3\gamma}A$ and $d_{\cb}(A,B)<6\gamma$.
\end{corollary}
\begin{proof}
Suppose $d(A,B)<(2+6\sqrt{2})^{-1}$ and $\tau:B\rightarrow\mathbb C$ is a bounded trace.  Let $\pi:B\rightarrow\Bb(\Hs)$ be the GNS-representation of $B$ corresponding to $\tau$. Then there is a larger Hilbert space $\tilde{\Hs}$ and a representation $\tilde{\pi}:C^*(A,B)\rightarrow\Bb(\tilde{\Hs})$ such that $\pi$ is a direct summand of $\tilde{\pi}|_B$. That is, the projection $p$ from $\tilde{\Hs}$ onto $\Hs$ is central in $\tilde{\pi}(B)$ and $\pi(b)=p\tilde{\pi}(b)p$ for all $b\in B$. Then, by \cite[Lemma 5]{KK:AJM}, we have $d(\tilde{\pi}(A)'',\tilde{\pi}(B)'')\leq d(A,B)$, and hence there is a projection $q\in \tilde{\pi}(A)''$ with $\|p-q\|\leq\gamma/\sqrt{2}$ by \cite[Lemma 1.10(ii)]{Kh:JOT}.  If $q$ is an infinite projection in $\tilde{\pi}(A)''$, then as $d(A,B)<(2+6\sqrt{2})^{-1}$, one can follow the argument of \cite[Lemma 6.1]{CSSW:GAFA} (using the estimate $\|p-q\|<\gamma/\sqrt{2}$ in place of $\|p-q\|<2\gamma$) to see that $p$ is infinite in $\tilde{\pi}(B)''$, giving a contradiction. If $q$ is finite, then $q\tilde{\pi}(A)''q$ has a finite trace $\rho$ and $\rho\circ\tilde{\pi}|_A$ defines a bounded trace on $A$, and again we have a contradiction. Thus $B$ has no bounded traces.

Theorem 2.4 of \cite{C:IJM} shows that a properly infinite von Neumann algebra has property $D_{3/2}^*$. As such, every $C^*$-algebra with no bounded traces has property $D_{3/2}$.   Since $A$ and $B$ both have property $D_{3/2}$, Proposition \ref{DKCB} gives $A\subset_{\cb,3\gamma}B$ and $B\subset_{\cb,3\gamma}A$, whence $d_{\cb}(A,B)<6\gamma$.
\end{proof}

\begin{corollary}\label{Cor:Stable}
\label{cor:stable}
Suppose that $A$ and $B$ are $C^*$-algebras which act non-degenerately on the same Hilbert space and satisfy $d(A,B)<1/252$ and suppose that $A$ has no bounded traces (for example when $A$ is stable).  Then $$(\Cu(A),\Sc(\Cu(A))\cong(\Cu(B),\Sc(\Cu(B)).$$
\end{corollary}
\begin{proof}
Combine Corollary \ref{NoTraces} with Theorem \ref{CuMap2} (noting that $6d(A,B)<1/42$).
\end{proof}

We can use the Cuntz semigroup to show that stability transfers to close $C^*$-algebras provided one algebra has stable rank one. To detect stability for a $\sigma$-unital $C^*$-algebra we use the following criterion from  \cite[Lemma 5.4]{OPR:IMRN} which reformulates the earlier characterisation from \cite{HR:JFA}.

\begin{lemma}
 \label{lem:stable} Let $A$ be a $\sigma$-unital $C^*$-algebra and let $c\in A$ be a strictly positive element. Then, $A$ is stable if and only if for every $\epsilon>0$, there is $b\in A^+$ such that $(c-\epsilon)_+\perp b$ and $(c-\epsilon)_+\precsim b$.
\end{lemma}

Following \cite{RS:JFA}, we say that a $C^*$-algebra $A$ has \emph{weak cancellation} provided $\Cu(A)$ satisfies the property that $x+z\ll y+z$ implies $x\leq y$. It was proved in \cite[Theorem 4.3]{RW:Crelle} that if $A$ has stable rank one, then $W(A)$ has the property defining weak cancelation. When $A$ has stable rank one, so too does $A\otimes\K$ \cite[Theorem 3.6]{R:PLMS}, and so $A$ has weak cancelation.

\begin{lemma}\label{lem:stable2} 
Let $A$ be a $\sigma$-unital $C^*$-algebra with weak cancelation.  Then $A$ is stable if and only if $\Cu(A)=\Sc(\Cu(A))$.
\end{lemma}
\begin{proof}
If $A$ is stable, then $\Sc(\Cu(A))=\Cu(A)$.  Indeed, given $n\in\mathbb N$, choose an automorphism $\theta_n:\K\otimes\K\rightarrow \K\otimes\K$ with $\theta_n(\K\otimes e_{11})=K\otimes M_n$ and let $\psi:A\rightarrow A\otimes\K$ be an isomorphism.   Then $(\psi^{-1}\otimes \id_\K)(\id_A\otimes \theta_n)(\psi\otimes \id_\K)$ is an automorphism of $A\otimes\K$ which maps $A\otimes e_{11}$ onto $A\otimes M_n$. In this way the class of a positive element in $A\otimes M_n$ lies in the scale $\Sc(\Cu(A))$.  For $x\in (A\otimes\K)_+$ and $\eps>0$, we have $(x-\eps)_+\in \bigcup_{n=1}^\infty(A\otimes M_n)$, and hence $\langle(x-\eps)_+\rangle\in \Sc(\Cu(A))$. Since the scale is defined to be closed under suprema, it follows that $\Cu(A)=\Sc(\Cu(A))$.

Conversely, let $c\in A$ be a strictly positive element so that $\Sc(\Cu(A))=\{x\in\Cu(A):x\leq\langle c\rangle\}$ and let $\epsilon>0$ be given. The hypothesis ensures that $2\langle c\rangle\leq \langle c\rangle$, and so we can find $\delta>0$ such that $2\langle (c-\tfrac{\epsilon}{4})_+\rangle\ll \langle (c-\delta)_+\rangle$. Now write $$
(c-\delta)_+= (c-\delta)_+g_{\epsilon/2,\epsilon}(c)+(c-\delta)_+(1_{M(A)}-g_{\epsilon/2,\epsilon}(c)),
$$
and observe that
$$
\langle (c-\delta)_+g_{\epsilon/2,\epsilon}(c)\rangle\leq\langle g_{\epsilon/2,\epsilon}(c)\rangle=\langle (c-\tfrac{\epsilon}{2})_+\rangle\ll \langle (c-\tfrac{\epsilon}{4})_+\rangle.
$$
We now have that 
$$
2\langle (c-\tfrac{\epsilon}{4})_+\rangle\ll \langle (c-\delta)_+\rangle\leq \langle (c-\tfrac{\epsilon}{2})_+\rangle+\langle (c-\delta)_+(1_{M(A)}-g_{\epsilon/2,\epsilon}(c))\rangle
$$
and so weak cancellation enables us to conclude that $$\langle (c-\tfrac{\epsilon}{4})_+\rangle\leq \langle (c-\delta)_+(1-g_{\epsilon/2,\epsilon}(c))\rangle.$$
Let $b=(c-\delta)_+(1-g_{\epsilon/2,\epsilon}(c))$. It is clear that $b\perp (c-\epsilon)_+$ and that $(c-\epsilon)_+\leq (c-\tfrac{\epsilon}{4})_+\precsim b$. Thus we may invoke Lemma \ref{lem:stable} to conclude that $A$ is stable.
\end{proof}

\begin{theorem}
Let $A$ and $B$ be $\sigma$-unital $C^*$-algebras with $A$ stable and $d(A,B)<1/252$ and suppose either $A$ or $B$ has stable rank one. Then $B$ is stable.
\end{theorem}
\begin{proof}
By Corollary \ref{Cor:Stable}, we have an isomorphism $$(\Cu(A),\Sigma(\Cu(A))\cong (\Cu(B),\Sigma(\Cu(B)).$$ Since $A$ is stable $\Cu(A)=\Sc(\Cu(A))$. Our isomorphism condition now tells us that $\Cu(B)=\Sigma(\Cu(B))$. If $B$ has stable rank one, then it has weak cancelation, whereas if $A$ has stable rank one, $A$ has weak cancelation and, as weak cancelation is a property of the Cuntz semigroup, so too does $B$.  The result now follows from Lemma \ref{lem:stable2}.
\end{proof}

We now turn to the situation in which one $C^*$-algebra is $\Z$-stable. In \cite{C:IJM}, Christensen shows that McDuff II$_1$ factors have property $D_{5/2}$, and hence via the estimates of \cite{P:StP}, have similarity length at most $5$. (In fact McDuff factors, and more generally II$_1$ factors with Murray and von Neumann's property $\Gamma$ have length $3$ \cite{C:JFA}, but at present we do not know how to use this fact to obtain better estimates for automatic complete closeness of close factors with property $\Gamma$.)

 In \cite{Pop:BAus,LS:arXiv,JW:BLMS} analogous results have been established in a $C^*$-setting: in particular $\Z$-stable $C^*$-algebras (\cite{JW:BLMS}) and $C^*$-algebras of the form $A\otimes B$, where $B$ is nuclear and has arbitrarily large unital matrix subalgebras (\cite{Pop:BAus}) have similarity degree (and hence length) at most $5$.  Here we show how to use the original von Neumann techniques from \cite{C:IJM} to show that a class of algebras generalising both these examples have property $D_{5/2}$ (recapturing the upper bound $5$ on the length).  A similar result has been obtained independently by Hadwin and Li \cite[Corollary 1]{HL:arXiv} working in terms of the similarity degree as opposed to property $D_k$. Once we have this $D_k$ estimate, Proposition \ref{DKCB} applies. In particular we obtain uniform estimates on the cb-distance $d_{\cb}(A,B)$ in terms of $d(A,B)$ when $A$ is $\Z$-stable.

Given a von Neumann algebra $\M\subset\Bb(\Hs)$ and $x\in \Bb(\Hs)$, write $\co_\M(x)$ for the weak$^*$ closed convex hull of $\{uxu^*:u\in\mathcal U(\M)\}$. If $\M$ is injective, then by Schwartz's property P, $\co_\M(x)\cap \M'$ is non-empty for all $x\in\Bb(\Hs)$. Note that for a non-degenerately represented $C^*$-algebra $A\subset\Bb(\Hs)$, we have $\|\ad(T)|_A\|=\|\ad(T)|_{A''}\|$.  We say that an inclusion $A\subset C$ of $C^*$-algebras is non-degenerate if the inclusion map is non-degenerate.

\begin{proposition}\label{DKProp}
Let $C$ be a $C^*$-algebra and $A,B\subset C$ be commuting non-degenerate $C^*$-subalgebras which generate $C$. Suppose $B$ is nuclear and has no non-zero finite dimensional representations.  Then $C$ has property $D_{5/2}$, and hence similarity length at most $5$.
\end{proposition}
\begin{proof}
Suppose $C$ is non-degenerately represented on $\Hs$ and fix $x\in\Bb(\Hs)$. The non-degeneracy assumption ensures that $A$ and $B$ are non-degenerately represented on $\Hs$. Note that $C''$ has no finite type I part as $B$ has no non-zero finite dimensional representations. Let $p$ be the central projection in $C''$ so that $C''p$ is type II$_1$ and $C''(1-p)$ is properly infinite. Fix a unital type I$_\infty$ subalgebra $\M_0\subset (1-p)C''(1-p)$ and let $\M=(\M_0\cup pB)''$ which is injective. By Schwartz's property P, there exists $y\in \co_{\M}(x)\cap (\M\cup\{p\})'$. As in \cite[Theorems 2.3, 2.4]{C:IJM}, $\|y-x\|\leq\|\ad(x)|_{C''}\|$ and $\|\ad(y)|_{C''}\|\leq\|\ad(x)|_{C''}\|$.  Write $y_1=yp$ and $y_2=y(1-p)$. If $p\neq 1$, then the properly infinite algebra $\M_0$ lies in $C''(1-p)\cap\{y_2,y_2^*\}'$ and so by \cite[Corollary 2.2]{C:IJM},  $\|\ad(y_2)|_{C''(1-p)}\|=2d(y_2,C'(1-p))$.  Take $x_2\in C''(1-p)$ with $\|x_2-y_2\|=\|\ad(y_2)|_{C''(1-p)}\|/2\leq \|\ad(x)|_{C''}\|/2$.

If $p\neq 0$, then we argue exactly as in the proof of \cite[Proposition 2.8]{C:IJM} to produce first $z_1\in A'p$ with $\|y_1-z_1\|\leq\|\ad(y_1)|_{C''p}\|/2\leq\|\ad(x)|_{C''}\|/2$.  Continuing with the proof of \cite[Proposition 2.8]{C:IJM}, as $B''p$ and $A''p$ commute, $\co_{B''p}(z_1)$ is contained in $A'p$ and hence there exists $x_1\in\co_{B''p}(z_1)\cap B'p$ with $$
\|x_1-z_1\|\leq\|\ad(z_1)|_{B''p}\|\leq \|\ad(z_1-y_1)|_{B''p}\|\leq 2\|z_1-y_1\|\leq\|\ad(x)|_{C''}\|.
$$
 Then 
$$
\|y_1-x_1\|\leq\|y_1-z_1\|+\|z_1-x_1\|\leq 3\|\ad(x)|_{C''}\|/2.
$$
If $p=0$, take $x_1=0$ and the same inequality holds.   The element $x_1+x_2\in C'$ has
\begin{align*}
&\|x-(x_1+x_2)\|\\
\leq&\|x-y\|+\|(y_1-x_1)+(y_2-x_2)\|\\
\leq&\|\ad(x)|_{C''}\|+\max(\|y_1-x_2\|,\|y_2-x_2\|)\leq 5\|\ad(x)_{C''}\|/2.
\end{align*}
Therefore $C$ has property $D_{5/2}$, and so by \cite[Remark 4.7]{P:StP} has length at most $5$.
\end{proof}

\begin{corollary}\label{Cor-ZStabDK}
Let $A$ be a $\Z$-stable $C^*$-algebra. Then $A$ has property $D_{5/2}$ and length at most $5$.
\end{corollary}

The main result of \cite{CSSW:GAFA} is that the similarity property transfers to close $C^*$-algebras.  This work is carried out with estimates depending on the length and length constant of $A$, but it is equally possible to carry out this work entirely in terms of property $D_k$ so it can be applied to $\Z$-stable algebras.  Our objective is to obtain a version of \cite[Corollary  4.6]{CSSW:GAFA} replacing the hypothesis that $A$ has length at most $\ell$ and length constant at most $K$ with the formally weaker hypothesis that $A$ has property $D_k$ (if $A$ has the specified length and length constants, then it has property $D_k$ for $k=K\ell/2$, conversely if $A$ has property $D_k$, then it has length at most $\lfloor 2k\rfloor$, but a length constant estimate is not known in this case, see \cite[Remark 4.7]{P:StP}).  This enables us to use Corollary \ref{Cor-ZStabDK} obtain an isomorphism between the Cuntz semigroups of sufficiently close $C^*$-algebras when one algebra is $\Z$-stable.  
To achieve a $D_k$ version of \cite[Section 4]{CSSW:GAFA}, we adjust the hypotheses in Lemma 4.1, Theorem 4.2 and Theorem 4.4 of \cite{CSSW:GAFA} in turn, starting with Lemma 4.1. We begin by isolating a technical observation.

\begin{lemma}\label{DKTech}
Let $\M$ be a finite von Neumann algebra with a faithful tracial state acting in standard form on $\Hs$ and let $J$ be the conjugate linear modular conjugation operator inducing an isometric antisomorphism $x\mapsto JxJ$ of $\M$ onto $\M'\cong \M^{\mathrm{op}}$. Suppose that $\mathcal S$ is another von Neumann algebra acting nondegenerately on $\Hs$ with $\M'\subset_\gamma \mathcal S$. If $\M$ has property $D_k^*$, then $\M'\subset_{\cb,2k\gamma}\mathcal S$.
\end{lemma}
\begin{proof}
As $J$ is isometric, $\M\subset_\gamma J\mathcal SJ$, so that $\M\subset_{\cb,2k\gamma}J\mathcal SJ$ by Proposition \ref{DKCB}.  Now, for each $n\in\mathbb N$, let $J_n$ denote the isometric conjugate linear operator of component wise complex conjugation on $\mathbb C^n$ so that $J\otimes J_n$ is a conjugate linear isometry  on $\Hs\otimes\mathbb C^n$. We can conjugate the near inclusion $\M\otimes M_n\subset_{2k\gamma}JSJ\otimes M_n$ by $J\otimes J_n$ to obtain $\M'\otimes M_n\subset_{2k\gamma}\mathcal S\otimes M_n$, as required. 
\end{proof}

The next lemma is the modification of \cite[Lemma 4.1]{CSSW:GAFA}. The expression for $\beta$ below is a slight improvement over that of the original.
\begin{lemma}\label{Mod1}
Let $\M$ and $\N$ be von Neumann algebras of type II$_1$ faithfully and non-degnerately represented on $\Hs$ with common centre $Z$ which admits a faithful state.  Suppose $d(\M,\N)=\alpha$ and $\M$ has property $D_k^*$. If $\alpha$ satisfies
$$
24(12\sqrt{2}k+4k+1)\alpha<1/200,
$$
then $d(\M',\N')<2\beta+1200k\alpha(1+\beta)$ where $\beta=96k\alpha(600k+1)$.
\end{lemma}
\begin{proof}
This amounts to showing that the hypothesis in \cite[Lemma 4.1]{CSSW:GAFA} that $\M$ contains an weak$^*$ dense $C^*$-algebra $A$ of length at most $\ell$ and length constant at most $K$ can be replaced by the statement that $\M$ has property $D_k^*$ (and that the specified expressions on $\beta$ are valid). The hypothesis that $\M$ has such a weak$^*$ dense $C^*$-algebra is initially used to see that $\M$ has property $D_k$ at the beginning of the lemma and then applied to a unital normal representation to obtain \cite[equation (4.5)]{CSSW:GAFA}. As such property $D_k^*$ suffices for this estimate.  

The other use of this hypothesis comes on page 385 in the last paragraph of the lemma, to obtain \cite[equation (4.28)]{CSSW:GAFA}.  Using the notation of this paragraph, the von Neumann algebra $T_\M$ is a cutdown of $\M$ acting as $\M\otimes I_\mathcal G$ on $\mathcal H\otimes\mathcal G$ by the projection $e_{i_0,i_0}$ from the commutant of $\M$ on this space. Since $e_{i_0,i_0}$ is unitarily equivalent in this commutant to a projection of the form $e\otimes g_0$, where $e$ is a projection from the commutant of $\M$ on $\Hs$ of full central support and $g_0$ is a minimal projection in $\mathcal B(\mathcal G)$, it follows that $e_{i_0,i_0}$ has full central support in the commutant of $\M$ on $\Hs\otimes\mathcal G$. As such $T_\M$ is isomorphic to $\M$, so has property $D_k^*$.  Thus Lemma \ref{DKTech} can be applied to the near inclusion $T_\M'\subset_{48(600k\alpha+\alpha)}T_{\N_2}'$ from \cite[equation (4.25)]{CSSW:GAFA} giving
$$
T_\M'\subset_{\cb,96k(600k\alpha+\alpha))}T_{\N_2}'.
$$
It then follows that 
$$
T_\M\,\overline{\otimes}\,\mathcal B(\ell^2(\Lambda))\subset_{96k(600k\alpha+\alpha)}T_\N\,\overline{\otimes}\,\mathcal B(\ell^2(\Lambda)),
$$
which is precisely \cite[equation (4.28)]{CSSW:GAFA} with our new estimate for $\beta$ replacing that of the original.  We then deduce that $d(\M',\N')\leq 2\beta+1200k\alpha(1+\beta)$ in just the same way that \cite[equation (4.30)]{CSSW:GAFA} is obtained from \cite[equation (4.28)]{CSSW:GAFA}.
\end{proof}

Now we adjust Theorem 4.2 of \cite{CSSW:GAFA}.  The resulting constant $\beta$ is obtained by taking $\alpha=11\gamma$ in the previous lemma. Note that there is an unfortunate omission in the value of $\beta$ in Theorem 4.2 of \cite{CSSW:GAFA} which should be given by taking $\alpha=11\gamma$ in Lemma 4.1 of \cite{CSSW:GAFA}, so should be $K((1+316800k\gamma+528\gamma)^\ell-1)$: this has no knock on consequences to Theorem 4.4 of \cite{CSSW:GAFA} where the correct value of $\beta$ is used.
\begin{lemma}\label{Mod2}
Let $A$ and $B$ be $C^*$-algebras acting on a Hilbert space and suppose that $d(A,B)=\gamma$. Suppose $A$ has property $D_k$ and $24(12\sqrt{2}k+4k+1)\gamma<1/2200$. Then 
$$
d(A',B')\leq 10\gamma+2\beta+13200k\gamma(1+\beta),
$$
where $\beta=1056k(600k\gamma+\gamma)$.
\end{lemma}
\begin{proof}
This amounts to replacing the hypothesis that $A$ has length at most $\ell$ and length constant at most $K$ with the condition that $A$ has property $D_k$ in Theorem 4.2 of \cite{CSSW:GAFA}.  The length hypothesis on $A$ is used to show that certain II$_1$ von Neumann closures of $A$ satisfy \cite[Lemma 4.1]{CSSW:GAFA}, but since the weak$^*$-closure of a $C^*$-algebra with property $D_k$ has property $D_k^*$, Lemma \ref{Mod1} can be used in place of \cite[Lemma 4.1]{CSSW:GAFA}.  Note that in the proof of \cite[Theorem 4.2]{CSSW:GAFA} the reference to injective von Neumann algebras having property $D_1$ is incorrect (it is an open question whether $\prod_{n=1}^\infty M_n$ has the similarity property). The correct statement is that these algebras have property $D_1^*$ which is all that is used.\end{proof}

Finally we can convert Theorem 4.4 of \cite{CSSW:GAFA}.  Note the typo in the statement of this theorem, the definition of $\tilde{k}$ should be $\frac{k}{1-2\eta-2k\gamma}$ rather than $\frac{k}{1-2\eta-k\gamma}$. The same change should be made in Corollary 4.6 of \cite{CSSW:GAFA}. 

\begin{proposition}\label{DKVer}
Let $A$ and $B$ be $C^*$-subalgebras of some $C^*$-algebra $C$ with $d(A,B)<\gamma$ and suppose that $A$ has property $D_k$.  Write $\beta=1056(600k\gamma+\gamma)$ and $\eta=10\gamma+2\beta+13200k\gamma(1+\beta)$ and suppose that 
\begin{equation}\label{DKVer1}
24(12\sqrt{2}k+4k+1)\gamma<\frac{1}{2200},\quad 2\eta+2k\gamma<1.
\end{equation}
Then $d_{\cb}(A,B)\leq 4\tilde{k}\gamma$, where
$$
\tilde{k}=\frac{k}{1-2\eta-2k\gamma}.
$$
\end{proposition}
\begin{proof}
We check that $B$ has property $D_{\tilde{k}}$. This amounts to weakening the hypothesis of \cite[Theorem 4.4]{CSSW:GAFA} in just the same way as the preceeding lemmas.  Applying Lemma \ref{Mod2} in place of Theorem 4.2 of \cite{CSSW:GAFA} in the proof of Theorem 4.4 of \cite{CSSW:GAFA}, shows that under the hypotheses of this proposition $B$ has property $D_{\tilde{k}}$, where
$$
\tilde{k}=\frac{k}{1-2\eta-2k\gamma}.
$$
This is valid as property $D_k$ descends to quotients so, following the proof of \cite[Theorem 4.4]{CSSW:GAFA}, the algebra $\rho(A)$ inherits property $D_k$ allowing the use of Lemma \ref{Mod2} above in place of \cite[Theorem 4.2]{CSSW:GAFA}. Note that one should take care with issues of degeneracy here.  In particular, the representation $\pi$ of $B$ in the proof of Theorem 4.4 of \cite{CSSW:GAFA} should be assumed non-degenerate.  

Proposition \ref{DKCB} now shows that $B\subset_{\cb,2\tilde{k}}A$ and $A\subset_{\cb,2k}B$.  Therefore $d_{\cb}(A,B)\leq 2\max(2\tilde{k}\gamma,2k\gamma)=4\tilde{k}\gamma$
\end{proof}

\begin{corollary}
Let $A$ be a $C^*$-algebra generated by two commuting non-degenerate $C^*$-subalgebras one of which is nuclear and has no finite dimensional irreducible representations. Suppose that $A\subset\Bb(\Hs)$ and $B$ is another $C^*$-subalgebra of $\Bb(\Hs)$ with $d(A,B)<\gamma$ for $\gamma<1/6422957$. Then $d_{\cb}(A,B)<1/42$ and $(\Cu(A),\Sc(\Cu(A))\cong (\Cu(B),\Sc(\Cu(B))$.
\end{corollary}
\begin{proof}
By Proposition \ref{DKProp}, $A$ has property $D_k$ for $k=5/2$ so in Proposition \ref{DKVer}, $\beta=1585056\gamma$ and $\eta=3203122\gamma+52306848000\gamma^2$ so that $2\eta+2k\gamma<10^{11}\gamma<1$ for $\gamma<10^{-11}$. The bound on $\gamma$ ensures that (\ref{DKVer1}) holds so that Proposition \ref{DKVer} applies. Further this bound gives
$$
\frac{4k\gamma}{1-2\eta-2k\gamma}<\frac{1}{42},
$$
and so the result follows from Proposition \ref{DKVer} and Theorem \ref{CuMap2}.
\end{proof}

In particular, $C^*$-algebras sufficiently close to $\Z$-stable algebras are automatically completely close and have the Cuntz semigroup of a $\Z$-stable algebra. The question of whether the property of $\Z$-stability transfers to sufficiently close subalgebras raised in \cite{CSSWW:Acta} remains open.
\begin{corollary}
Let $A$ be a $\Z$-stable $C^*$-algebra and suppose that $B$ is another $C^*$-algebra acting on the same Hilbert space as $A$ with $d(A,B)<1/6422957$. Then $d_{\cb}(A,B)<1/42$, $(\Cu(A),\Sc(\Cu(A))\cong (\Cu(B),\Sc(\Cu(B))$. In particular $B$ has the Cuntz semigroup of a $\Z$-stable algebra.
\end{corollary}

\section{Quasitraces}\label{Sect5}

In this section we use our isomorphism between the Cuntz semigroups of completely close $C^*$-algebras to give an affine homeomorphism between the lower semicontinuous quasitraces on such algebras. This isomorphism is compatible with the affine isomorphism of the trace spaces of close $C^*$-algebras constructed in \cite[Section 5]{CSSW:GAFA}. 

Given a $C^*$-algebra $A$, write $T(A)$ for the cone of lower semicontinuous traces on $A$ and $QT_2(A)$ for the cone of lower semicontinuous $2$-quasitraces on $A$.  Precisely, a trace $\tau$ on $A$ is a linear function $\tau:A_+\rightarrow [0,\infty]$ vanishing at $0$ and satisfying the trace identity $\tau(xx^*)=\tau(x^*x)$ for all $x\in A$.  A $2$-quasitrace is a function $\tau:A_+\rightarrow [0,\infty]$ vanishing at $0$ which satisfies the trace identity and which is linear on commuting elements of $A_+$. Write $T_s(A)$ for the simplex of tracial states on $A$ and $QT_{2,s}(A)$ for the bounded $2$-quasitraces on $A$ of norm one.  Lower semicontinuous traces and $2$-quasitraces on $A$ extend uniquely to lower semicontinuous traces and $2$-quasitraces respectively on $A\otimes\K$ (see \cite[Remark 2.27(viii)]{BK:JFA}).

 In \cite[Section 4]{ERS:AJM}, Elliott, Robert and Santiago extend earlier work of Blackadar and Handelman, setting out how functionals on $\Cu(A)$ arise from elements of $QT_2(A)$. Precisely a functional on $\Cu(A)$ is a map $f:\Cu(A)\rightarrow[0,\infty]$ which is additive, order preserving, has $f(0)=0$ and preserves the suprema of increasing sequences. Given $\tau\in QT_2(A)$, the expression $d_\tau(\langle a\rangle)=\lim_{n\rightarrow\infty}\tau(a^{1/n})$ gives a well defined functional on $\Cu(A)$, where we abuse notation by using $\tau$ to denote the extension of the original lower semicontinuous $2$-quasitrace to $A\otimes\K$. Alternatively, one can define $d_\tau$ by $d_\tau(\langle a\rangle)=\lim_{n\rightarrow\infty}\tau(a_n)$, where $(a_n)_{n=1}^\infty$ is any very rapidly increasing sequence from $(A\otimes\K)^+_1$ representing $\langle a\rangle$. Conversely, given a functional $f$ on $\Cu(A)$, a lower semicontinuous $2$-quasitrace on $A\otimes\K$ (and hence on $A$) is given by $\tau_f(a)=\int_0^\infty f(\langle (a-t)_+\rangle)\mathrm{d}t$. With this notation, the assignments $\tau\mapsto d_\tau$ and $f\mapsto \tau_f$ are mutually inverse (see \cite[Proposition 4.2]{ERS:AJM}).

The topology on $QT_2(A)$ is specified by saying that a net $(\tau_i)$ in $QT_2(A)$ converges to $\tau\in QT_2(A)$ if and only if 
$$
\limsup_i \tau_i((a-\eps)_+)\leq\tau(a)\leq\liminf_i\tau_i(a)
$$
for all $a\in A_+$ and $\eps>0$.  With this topology $QT_2(A)$ is a compact Hausdorff space \cite[Theorem 4.4]{ERS:AJM} and $T(A)$ is compact in the induced topology \cite[Theorem 3.7]{ERS:AJM}.  In a similar fashion, the cone of functionals on $\Cu(A)$ is topologised by defining $\lambda_i\rightarrow \lambda$ if and only if 
$$
\limsup_i\lambda_i(\langle(a-\eps)_+\rangle)\leq\lambda(\langle a\rangle)\leq\liminf_i\lambda_i(\langle a\rangle)
$$
for all $a\in (A\otimes\K)_+$ and $\eps>0$.  Theorem 4.4 of \cite{ERS:AJM} shows that the affine map $\tau\mapsto d_\tau$ is a homeomorphism between the cone $QT_2(A)$ and the cone of functionals on the Cuntz semigroup.

\begin{theorem}\label{QT}
\begin{enumerate}
\item\label{QT.1} Let $A,B$ be $C^*$-algebras acting non-degenerately on a Hilbert space with $d_{\cb}(A,B)<1/42$. Then the isomorphism $$\Phi:(\Cu(A),\Sc(\Cu(A)))\rightarrow(\Cu(B),\Sc(\Cu(B)))$$ given by Theorem \ref{CuMap2} induces an affine homeomorphism $$\widehat{\Phi}:QT_2(B)\rightarrow QT_2(A)$$ satisfying
\begin{equation}\label{QT.E}
d_{\hat{\Phi}(\tau)}(x)=d_\tau(\Phi(x))
\end{equation}
for all $x\in\Cu(A)$ and $\tau\in QT_2(B)$.
\item\label{QT.2} Suppose additionally that $A$ and $B$ are unital and $d_\cb(A,B)<\gamma<1/2200$.  Then the map $\hat{\Phi}$ from (\ref{QT.1}) is compatible with the map $\Psi:T_s(B)\rightarrow T_s(A)$ given in Lemma 5.4 of \cite{CSSW:GAFA}. Precisely, for $\tau\in T_s(B)$, we have $\widehat{\Phi}(\tau)\in T_s(A)\subset QT_2(A)$ and $\widehat{\Phi}(\tau)=\Psi(\tau)$.
\end{enumerate}
\end{theorem}

\begin{proof}
The first part of the theorem is a consequence of Theorem \ref{CuMap2} and \cite[Proposition 4.2]{ERS:AJM}: given $\tau\in QT_2(B)$, define $\widehat{\Phi}(\tau)$ to be the lower semicontinuous $2$-quasitrace induced by the functional $d_\tau\circ\Phi$ on $\Cu(A)$. It is immediate from the construction that the map $\widehat{\Phi}$ is affine, bijective and the identity (\ref{QT.1}) holds.

To show that $\widehat{\Phi}$ is continuous, we use the homeomorphism between the cone of lower semicontinuous quasi-traces and functionals on the Cuntz semigroup in \cite[Theorem 4.4]{ERS:AJM}.  Consider a net $(\tau_i)$ in $QT_2(B)$ with $\tau_i\rightarrow \tau$.  Fix $a\in A_+$, then, 
$$
d_{\tau}(\Phi(\langle a\rangle))\leq \liminf_i d_{\tau_i}(\Phi(\langle a\rangle)),
$$
as $d_{\tau_i}\rightarrow d_{\tau}$. Now take $\eps>0$ and fix a contraction $b\in (B\otimes K)_+$ with $\Phi(\langle a\rangle)=\langle b\rangle$. As $(\langle (b-1/n)_+\rangle)_{n=1}^\infty$ is very rapidly increasing with supremum $\langle b\rangle$, there exists $n\in\mathbb N$ with $\Phi(\langle (a-\eps)_+\rangle)\leq \langle (b-1/n)_+\rangle$.  As
$$
\limsup_id_{\tau_i}(\langle b-1/n)_+\rangle)\leq d_\tau(\langle b\rangle),
$$
it follows that
$$
\limsup_id_{\widehat{\Phi}(\tau_i)}(\langle (a-\eps)_+\rangle)\leq d_{\widehat{\Phi}(\tau)}(\langle a\rangle)\leq\liminf_id_{\widehat{\Phi}(\tau)}(\langle a\rangle).
$$
Thus $d_{\widehat{\Phi}(\tau_i)}\rightarrow d_{\widehat{\Phi}(\tau)}$ and so, using the homeomorphism between $QT_2(A)$ and functionals on $\Cu(A)$, we have $\widehat{\Phi}(\tau_i)\rightarrow\widehat{\Phi}(\tau)$.  Therefore $\widehat{\Phi}$ is continuous, and hence a homeomorphism between $QT_2(B)$ and $QT_2(A)$.

For the second part we first need to review the construction of the map $\Psi$ from \cite{CSSW:GAFA}.  Suppose $d_{\cb}(A,B)<\gamma<1/2200$. Write $C=C^*(A,B)$ and let $C\subset\Bb(\Hs)$ be the universal representation of $C$ so that $\M=A''$ and $\N=B''$ are isometrically isomorphic to $A^{**}$ and $B^{**}$ respectively. Note that the Kaplansky density argument of \cite[Lemma 5]{KK:AJM} gives $d_\cb(\M,\N)\leq d_{\cb}(A,B)$. Following the proof of \cite[Lemma 5.4]{CSSW:GAFA} we can find a unitary $u\in (Z(\M)\cup Z(\N))''$ such that $Z(u\M u^*)=Z(\N)$ and $\|u-1_C\|\leq 5\gamma$. We write $A_1=uAu^*$ and $\M_1=u\M u^*$. There is now a projection $z_{\text{fin}}\in Z(\M_1)=Z(\N)$ which simultaneously decomposes $\M_1=\M_1z_{\text{fin}}\oplus \M_1(1-z_{\text{fin}})$ and $\N=\N z_{\text{fin}}\oplus \N(1-z_{\text{fin}})$ into the finite and properly infinite parts respectively (\cite[Lemma 3.5]{CSSW:GAFA} or \cite{KK:AJM}). Given a tracial state $\tau$ on $B$, there is a unique extension $\tau''$ to $\N$, which then factors uniquely through the centre valued trace $\mathrm{Tr}_{\N z_{\text{fin}}}$ on $\N z_{\text{fin}}$. That is, $\tau''(x)=(\phi_\tau\circ\mathrm{Tr}_{\N z_{\text{fin}}})(xz_{\text{fin}})$ for some state $\phi_\tau$ on $\N z_{\text{fin}}$. The map $\Psi$ in \cite{CSSW:GAFA} is then given by defining $\Psi(\tau)(y)=(\phi_\tau\circ \mathrm{Tr}_{\M_1 z_{\text{fin}}})(uyu^*z_{\text{fin}})$ for $y\in A$.

Now fix $\tau\in T_s(B)$.  For $m\in\mathbb N$ and $a\in (A\otimes M_m)^+_1$, consider the standard very rapidly increasing sequence $(g_{2^{-(n+1)},2^{-n}}(a))_{n=1}^\infty$ which represents $\langle a\rangle$. Let $p_n\in \M\otimes M_m$ be the spectral projection for $a$ for $[2^{-(n+1)},1]$, so that the alternating sequence
$$
g_{2^{-2},2^{-1}}(a),p_1,g_{2^{-3},2^{-2}}(a),p_2,g_{2^{-4},2^{-3}}(a),p_3,\dots
$$
is very rapidly increasing. Then 
\begin{equation}\label{QT.E2}
d_{\Psi(\tau)}(\langle a\rangle)=\sup_n(\Psi(\tau))(g_{2^{-n},2^{-(n+1)}}(a))=\sup_n\Psi(\tau)''(p_n).
\end{equation}

Choose $b_n\in (B\otimes M_m)^+_1$ with $\|g_{2^{-(n+1)},2^{-n}}(a)-b_n\|\leq 2\gamma$ and projections $q_n\in \N\otimes M_m$ with $\|p_n-q_n\|\leq 2\gamma$ (by a standard functional calculus argument, see \cite[Lemma 2.1]{C:JLMS}).  Note that $d_\cb(\M_1,\N)\leq 11\gamma$ and the algebras $(\M_1\otimes M_m)(z_{\text{fin}}\otimes 1_m)$ and $(\N_1\otimes M_m)(z_{\text{fin}}\otimes 1_m)$ have the same centre. Since $\|(u\otimes 1_m)p(u\otimes  1_m)^*(z_{\text{fin}}\otimes 1_m)-q(z_{\text{fin}}\otimes 1_m)\|<1/2$, Lemma 3.6 of \cite{CSSW:GAFA} applies to show that 
$$
(\mathrm{Tr}_{\M_1 z_{\text{fin}}}\otimes\mathrm{tr}_m)((u\otimes 1_m)p_n(u\otimes 1_m)^*(z_{\text{fin}}\otimes 1_m))=(\mathrm{Tr}_{\N z_{\text{fin}}}\otimes\mathrm{tr}_m)(q(z_{\text{fin}}\otimes 1_m)).
$$
This ensures that $\Psi(\tau)''(p_n)=\tau''(q_n)$ for all $n$.  

As each $(q_n-18\gamma)_+=q_n$, the sequence
$$
(b_1-18\gamma)_+,q_1,(b_2-18\gamma)_+,q_2,(b_3-18\gamma)_+,q_3,\dots
$$
is upwards directed by Lemma \ref{Lem:Claim} and the supremum of this sequence defines $\Phi(\langle a\rangle)$.  We then have 
\begin{equation}\label{QT.E1}
d_\tau(\Phi(\langle a\rangle))=\sup_n\tau''(q_n).
\end{equation}
Indeed, $d_\tau(\Phi(\langle a\rangle))$ is given by $\sup\tau(c_n)$, where $(c_n)_{n=1}^\infty$ is any very rapidly increasing sequence in $(B\otimes \K)_+$ representing $\Phi(\langle a\rangle)$. But, working in $\Cu(\N)$, Proposition \ref{L:VRI2} shows that any such very rapidly increasing sequence $(c_n)_{n=1}^\infty$ can be intertwined with the very rapidly increasing sequence $(q_n)_{n=1}^\infty$ after telescoping, and this establishes (\ref{QT.E1}).   Combining (\ref{QT.E2}) and (\ref{QT.E1}), we have
\begin{equation}\label{QT.E3}
d_{\Psi(\tau)}(\langle a\rangle)=d_{\widehat{\Phi}(\tau)}(\langle a\rangle)
\end{equation}
for all $m\in\mathbb N$ and $a\in (A\otimes M_m)_+$. As functionals on the Cuntz-semigroup preserve suprema, (\ref{QT.E3}) holds for all $a\in (A\otimes \K)_+$, whence $\Psi(\tau)=\widehat{\Phi}(\tau)$.
\end{proof}

The homeomorphism between the lower semicontinuous quasi-traces can be used to establish the weak$^*$-continuity of the map between the tracial state spaces of close unital $C^*$-algebras from \cite[Section 5]{CSSW:GAFA} resolving a point left open there. In particular this shows that the map defined in \cite{CSSW:GAFA} provides an isomorphism between the Elliott invariants of completely close algebras, as a priori the

 For any closed two-sided ideal $I\unlhd A$, the  subcone $T_I(A)$ of $T(A)$ consists of those $\tau\in T(A)$ such that the closed two-sided ideal generated by $\{x\in A_+:\tau(x)<\infty\}$ is $I$.  Proposition 3.11 of \cite{ERS:AJM} shows that the relative topology on $T_I(A)$ is the topology of pointwise convergence on the positive elements of the Pedersen ideal of $I$. In particular, $T_s(A)\subset T_A(A)$. In particular, the induced topology on $T_s(A)$ is just the weak$^*$-topology.

\begin{corollary}
Suppose that $A$ and $B$ are unital $C^*$-algebras acting non-degenerately on a Hilbert space with $d_{\cb}(A,B)<1/42$ and $d(A,B)<1/2200$. Then the affine isomorphism $\Psi:T_s(B)\rightarrow T_s(A)$ between tracial state spaces in \cite[Section 5]{CSSW:GAFA} is a homeomorphism with respect to the weak$^*$-topologies.
\end{corollary}

We end with two further corollaries of Theorem \ref{QT}.

\begin{corollary}
Let $A$ and $B$ be unital $C^*$-algebras acting non-degenerately on the same Hilbert space with $d_{\cb}(A,B)<1/2200$. Suppose every bounded $2$-quasitrace on $A$ is a trace, then the same property holds for $B$.
\end{corollary}
\begin{proof}
Given $\tau\in QT_{2,s}(B)$, its image $\widehat{\Phi}(\tau)$ lies in $QT_{2,s}(A)=T_s(A)$. By Theorem \ref{QT} (\ref{QT.2}) (applied with $A$ and $B$ interchanged) 
$$
\tau=\widehat{\Phi}^{-1}(\widehat{\Phi}(\tau))=\Psi^{-1}(\widehat{\Phi}(\tau))\in T_s(B),
$$
as claimed.
\end{proof}

The question of whether exactness transfers to (completely) close $C^*$-algebras raised in \cite{CSSW:GAFA} remains open, but we do at least obtain the following corollary.
\begin{corollary}
Let $A$ and $B$ be unital $C^*$-algebras acting non-degenerately on the same Hilbert space with $d_{\cb}(A,B)<1/2200$ and suppose $A$ is exact. Then every bounded $2$-quasitrace on $B$ is a trace.
\end{corollary}
\begin{proof}
This is immediate from Haagerup's result that bounded $2$-quasitraces on exact $C^*$-algebras are traces (\cite{H:Pre}) and the previous corollary.
\end{proof}

We end by noting that the isomorphism between the Cuntz semigroups of completely close algebras in Theorem \ref{CuMap2} can also be used to directly recapture an isomorphism between the Elliott invariants in significant cases.  Let $\CuT$ be the functor $A\mapsto \Cu(A\otimes C(\mathbb T))$ mapping the category of $C^*$-algebras into the category ${\mathbf{Cu}}$ introduced in \cite{CEI:Crelle} and let $\Ell$ be the Elliott invariant functor taking values in the category $\Inv$ whose objects are the $4$-tuples arising from the Elliott invariant. Let $\mathcal C$ be the subcategory of separable, unital, simple finite and $\Z$-stable algebras $A$ with $QT_2(A)=T(A)$ (for example if $A$ is exact). Then, building on work from \cite{BPT:Crelle,BT:IMRN}, Theorem 4.2 of \cite{ADPS:TAMS} provides functors $F:\Inv\rightarrow \mathbf{Cu}$ and $G:\mathbf{Cu}\rightarrow\Inv$ such that there are natural equivalences of functors $F\circ\Ell|_{\mathcal C}\cong\CuT|_{\mathcal C}$ and $G\circ\CuT|_{\mathcal C}\cong\Ell|_{\mathcal C}$ (a similar result for simple unital ASH algebras which are not type I and have slow dimension growth can be found in \cite{T:IJM}). Note that in Theorem 4.2 there is an implicit nuclearity hypothesis, which is only actually used in order to see $QT_2(A)=T(A)$ --- the result holds in the generality stated. Thus if $A$ and $B$ are $\Z$-stable $C^*$-algebras with $d_{cb}(A,B)$ sufficiently small, and $A$ is simple, separable, unital finite and has $QT_2(A)=T(A)$, then $B$ enjoys all these properties. Further, since tensoring by an abelian algebra does not increase the complete distance between $A$ and $B$ (see \cite[Theorem 3.2]{C:Acta} for this result in the context of near inclusions --- the same proof works for the metric $d_{cb}$), $\CuT(A)\cong\CuT(B)$ by Theorem \ref{CuMap2}. Thus $\Ell(A)\cong \Ell(B)$.

\subsection*{Acknowledgements} This work was initiated at the Centre de Recerca Matem\` atica (Bellaterra) during the Programme ``The Cuntz Semigroup and the Classification of $C^ *$-algebras'' in 2011. The authors would like to thank the CRM for the financial support for this programme and the conducive research environment.

\end{document}